\documentclass[12pt]{extarticle}
\usepackage{amsmath, amsthm, amssymb, color}
\usepackage{graphicx}
\usepackage{caption}
\usepackage{mathtools}
\usepackage{enumitem}
\usepackage{verbatim}
\usepackage{tikz}
\usepackage{url}
\usetikzlibrary{matrix}
\usetikzlibrary{arrows}
\usepackage{algorithm}
\usepackage[noend]{algpseudocode}
\usepackage{caption}
\usepackage[normalem]{ulem}
\usepackage{subcaption}
\usepackage{algorithm}
\usepackage{algpseudocode}
\oddsidemargin \evensidemargin
\usepackage{makecell}
\usepackage{multirow,array}
\usepackage[capitalize, noabbrev]{cleveref}

\newtheorem{theorem}{Theorem}
\newtheorem{proposition}[theorem]{Proposition}
\newtheorem{lemma}[theorem]{Lemma}
\newtheorem{corollary}[theorem]{Corollary}

\theoremstyle{definition}

\newtheorem{remark}[theorem]{Remark}
\newtheorem{conjecture}[theorem]{Conjecture}

\newtheorem{example}[theorem]{Example}

\newcommand{\PP}{\mathbb{P}}

\newcommand{\ZZ}{\mathbb{Z}}
\newcommand{\NN}{\mathbb{N}}
\newcommand{\KK}{\mathbb{K}}

\title{\bf Subspaces Fixed by a Nilpotent Matrix}
\author{
  Marvin Anas Hahn, Gabriele Nebe, Mima Stanojkovski, and Bernd Sturmfels}

\date{ }
   
\begin{document}
\maketitle

\begin{abstract} \noindent
The linear spaces that are fixed by a given nilpotent $n \times n$ matrix
form a subvariety of the Grassmannian. 
We classify these varieties for small $n$.
Mutiah, Weekes and Yacobi conjectured that their
radical ideals are
generated by certain linear forms known as shuffle equations.
We prove this conjecture for $n \leq 7$, and we disprove it for $n=8$.
The question remains open
for nilpotent matrices arising from the affine Grassmannian.
\end{abstract}

\section{Introduction}

For an arbitrary field $K$, the Grassmannian ${\rm Gr}(\ell,n)$ parametrizes
 $\ell$-dimensional  subspaces $L$ of the vector space $K^n$.
Given any  matrix $T \in K^{n \times n}$, we write $LT$ for the image of $L$ under the 
map given by $T$. This right action is compatible with representing $L$ as the
row space of an $\ell \times n$ matrix ${\bf L}$.
The Pl\"ucker embedding of ${\rm Gr}(\ell,n) $ into $ \PP^{\binom{n}{\ell}-1}$
arises by representing $L$ with the vector of maximal minors $p_{i_1 i_2 \cdots i_\ell}$ of ${\bf L}$.
Its homogeneous prime ideal has a natural Gr\"obner basis of quadrics
\cite[Theorem 3.1.7]{AIT}. These are  known as the {\em Pl\"ucker quadrics}.

 In this paper we  assume that $T$ is nilpotent,
i.e.~$T^n = 0$, and we study the subvariety
$$ {\rm Gr}(\ell,n)^T \,\,= \,\, \bigl\{ L \in {\rm Gr}(\ell,n)\,:\, L T \subseteq L  \bigr\}. $$
We are interested in its homogeneous radical ideal in the Pl\"ucker coordinates $p_{i_1 i_2 \cdots i_\ell}$.

\begin{example}[$n=4,\ell=2$] \label{ex:42}
Fix a nonzero scalar $\epsilon $ and consider the nilpotent $4 \times 4$ matrix 
$$ T \,\, =\,\,\begin{small} \begin{pmatrix} 0 & 1  & 0 & 0 \\
0 & 0 & 0 & 0 \\
0 & 0 & 0 & \epsilon \\
0 & 0 & 0 & 0 \end{pmatrix} \end{small}. $$
The  fixed point locus ${\rm Gr}(2,4)^T$ is a singular surface in the
$4$-dimensional Grassmannian ${\rm Gr}(2,4) = V( p_{12} p_{34} - p_{13}p_{23} + p_{14} p_{23})$.
It is the quadratic cone in $\PP^3$ defined by the
prime~ideal 
\begin{equation}
\label{eq:42shuffle} \langle \,p_{13},\,
p_{14} + \epsilon \,p_{23},\,
p_{12} p_{34} - \epsilon p_{23}^2\,
\rangle  \quad = \quad
\langle \,p_{13},\,
p_{14} + \epsilon\, p_{23}\,\rangle  \, + \, \hbox{ideal of ${\rm Gr}(2,4)$}.
\end{equation}
On an affine chart of ${\rm Gr}(2,4)$, each plane $L$ that is fixed by $T$ is the
row span of a matrix
$$ \qquad {\bf L} \,\, = \,\, \begin{pmatrix} 
1 & 0 & x & y \\
0 & 1 & 0 & \epsilon  x 
\end{pmatrix}, \quad {\rm or} \quad
{\bf L} \,\, = \,\,
\begin{pmatrix} 
\epsilon z & w & 1 & 0 \\
0 & z &  0 & 1 \end{pmatrix}\,
\,\,\,\hbox{after setting}\,\,\, x = \frac{1}{\epsilon z} , \,y = - \frac{w}{\epsilon z^2}.
$$
Next consider the special case $\epsilon = 0$. The ideal
(\ref{eq:42shuffle}) is still radical, but it now decomposes:
\begin{equation}
\label{eq:42shuffledeco}
\langle \,p_{13}\,,\,
p_{14} \,,\,
p_{12} p_{34} \,
\rangle  \quad = \quad
\langle \,p_{13}, p_{14}, p_{34}\, \rangle  \,\,\cap\,\,
\langle \,p_{12},p_{13},p_{14} \,\rangle.
\end{equation}
The quadratic cone degenerates into two planes $\PP^2$ in ${\rm Gr}(2,4) \subset \PP^5$.
They are given by
$$
 {\bf L} \,\, = \,\, \begin{pmatrix} 
1 & 0 & x & y \\
0 & 1 & 0 & 0
\end{pmatrix}
\quad {\rm and} \quad
 {\bf L} \,\, = \,\, \begin{pmatrix} 
0 & w & 1 & 0 \\
0 & z & 0 & 1 
\end{pmatrix} .
$$
We conclude that ${\rm Gr}(2,4)^T$ can be singular or reducible.
For all values of $\epsilon \in K$,
its radical ideal is generated by two linear forms plus the Pl\"ucker 
quadric $ p_{12} p_{34} - p_{13}p_{23} + p_{14} p_{23}$.
\hfill $\diamond$ \end{example}

We assume from now on that the nilpotent matrix $T$ is 
in Jordan canonical form. The necessary change of basis in $K^n$
works over an arbitrary field $K$ because all the eigenvalues of $T$ are zero.
The matrix $T$ in Example \ref{ex:42} is in Jordan canonical form
when $\epsilon=0$ or $\epsilon=1$.

Kreiman, Lakshmibai, Magyar, and Weyman \cite{KLMW}
identified a natural set of linear forms in Pl\"ucker coordinates
that vanish on ${\rm Gr}(\ell,n)^T$.
These are called {\em shuffle equations} and they generalize
the two linear forms  seen in (\ref{eq:42shuffle}).
It was conjectured in \cite{KLMW}
 that the shuffle equations cut out certain models of the affine Grassmannian.
 Muthiah, Weekes and Yacobi \cite{MuWeYa} gave
 a reformulation of the shuffle equations, and they proved
 the main conjecture of \cite{KLMW}. We refer to \cite[Section 6]{MuWeYa}
 for that proof and for a conceptual discussion of the shuffle equations.
It was subsequently conjectured in \cite[Section 7]{MuWeYa} 
that the shuffle equations plus the Pl\"ucker quadrics
generate the radical ideal of ${\rm Gr}(\ell,n)^T$.
The present paper settles that conjecture.

Our presentation is organized as follows.
In Section \ref{sec2} we review the shuffle equations 
and we show how to generate them in {\tt Macaulay2} \cite{M2}.
   The duality result  in Theorem \ref{th:duality} 
     allows us to swap $\ell$ and $n-\ell$ in these computations.
In Section \ref{sec3} we present the classification of all
varieties ${\rm Gr}(\ell,n)^T$ for $n \leq 8$. We compute their
dimensions, degrees, irreducible components, and defining equations.
We disprove the conjecture of Muthiah, Weekes and Yacobi
 \cite[Conjecture 7.6]{MuWeYa} for $n=8$, and we show that it holds for $n \leq 7$.
    Section \ref{sec4} is devoted to finite-dimensional models of the
affine Grassmannian. Here $T$ is the nilpotent matrix
given by a partition of rectangular shape. We prove that
 ${\rm Gr}(\ell,n)^T$ is irreducible for such $T$, and we give a
matrix parametrization. We believe that 
Conjecture 7.1 in \cite{MuWeYa}~holds.
This is equivalent to \cite[Conjecture 7.6]{MuWeYa} for 
rectangular shapes. We offer supporting evidence.

\section{Shuffle Equations}
\label{sec2}

Fix a nilpotent $n \times n$ matrix $T = T_\lambda$ in Jordan canonical form.
Here $\lambda = (\lambda_1 \geq \lambda_2 \geq \cdots \geq \lambda_s)$ 
is any partition of the integer $n$. Each entry of the matrix $T_\lambda$ is either $0$ or $1$.
The entries~$1$ are located in positions $(j,j+1)$ where $j \in \{1,2,\ldots,n-1\}
\backslash \{ \lambda_1,\lambda_1+\lambda_2,\ldots,\lambda_1+ \cdots + \lambda_{s-1}\}$. In other words, 
$T_\lambda$ is the nilpotent matrix in Jordan canonical form where the sizes of the Jordan blocks are given by the 
parts $\lambda_i$ of the partition $\lambda$.
The rank of $T_\lambda$ equals $n-s$. We regard ${\rm ker}(T_\lambda) $
as a linear subspace of dimension $s-1$ in the projective space $\PP^{n-1}$.

The shuffle relations are defined as follows. 
Consider the  $n \times n$ matrix $\,{\rm Id}_n + z T\,$ where $z$ is a parameter.
 For $z \in K$, this is an automorphism of
the vector space $K^n$. A subspace $L$ of $K^n$ satisfies $LT \subseteq L$
if and only if $\,L \, ({\rm Id}_n+zT)  \,=\, L\,$ for all $z$.  
Writing $P \in K^{\binom{n}{\ell}}$ for the row vector of Pl\"ucker coordinates of $L$,
the last equation is equivalent to the identity
\begin{equation} \label{eq:wedge1} P \cdot \wedge_\ell ({\rm Id}_n+z T)  \,\,= \,\,P . \end{equation}
Here $\wedge_\ell ({\rm Id}_n+zT)$ is the $\ell$th exterior power of the $n \times n$ matrix ${\rm Id_n}+zT$.
This is an $\binom{n}{\ell} \times \binom{n}{\ell}$ matrix whose entries are polynomials
in $\ZZ[z]$  of degree $\leq \ell$. Equivalently, we can write
\begin{equation}
\label{eq:wedgematrix}
 \wedge_\ell ({\rm Id}_n+z T) \,\,\, = \,\,\, \wedge_\ell \,{\rm Id}_n \,\,+ \, \,
\sum_{i=1}^\ell\, [\,\wedge_\ell ({\rm Id}_n+z T)\,]_i\, z^i, 
\end{equation}
where the coefficient $[\,\wedge_\ell ({\rm Id}_n+z T)\,]_i$ of $z^i$ is an integer matrix of format $\binom{n}{\ell} \times \binom{n}{\ell}$.
From (\ref{eq:wedge1}) we then obtain
\begin{equation} \label{eq:wedge2} P \cdot  [\,\wedge_\ell ({\rm Id}_n+zT)\,]_i \,\, = \,\, 0
\quad {\rm for} \,\,\, i =1,2,\ldots,\ell.
\end{equation}
This is a finite collection of linear forms in the $\binom{n}{\ell}$ 
Pl\"ucker coordinates $p_{i_1 i_2 \cdots i_\ell}$. 
These are the {\em shuffle equations} of $T$. 
The following was proved by Muthiah et al.~in~\cite[Proposition~6.6]{MuWeYa}.

\begin{proposition}
The variety $\,{\rm Gr}(\ell,n)^T$ is the intersection of
the Grassmannian ${\rm Gr}(\ell,n)$ with a linear subspace in $\,\PP^{\binom{n}{\ell}-1}$.
That linear subspace is defined by the shuffle equations.
\end{proposition}

\begin{example}[$n=4,\ell=2$]  \label{ex:drei}
We compute the shuffle equations for the 
matrix $T$ in Example~\ref{ex:42}.
Write $P = (p_{12}, p_{13},p_{23}, p_{14},p_{24},p_{34})$.
With this ordering of the Pl\"ucker coordinates, we have
$$
\wedge_2 ({\rm Id}_4+z T)  \quad = \quad \begin{small}
\begin{pmatrix}
1 & 0 & 0 & 0 & 0 & 0 \\
0 & 1 & z & \epsilon z & \epsilon z^2 & 0 \\
0 & 0 & 1 & 0 & \epsilon z & 0 \\
0 & 0 & 0 & 1 & z & 0 \\
0 & 0 & 0 & 0 & 1 & 0 \\
0 & 0 & 0 & 0 & 0 & 1 
\end{pmatrix}. \end{small}
$$
In (\ref{eq:wedge2}), we find
$ P \cdot  [\,\wedge_2 ({\rm Id}_4+zT)\,]_1 \,=\,
\bigl(0, 0, p_{13} , \epsilon p_{13} ,
(\epsilon p_{23} + p_{14}) ,0 \bigr) \,$ and $\,
P \cdot  [\,\wedge_2 ({\rm Id}_4+zT)\,]_2 \, = \,
\bigl(\,0,\,0,\,0,\,0, \,\epsilon p_{13} ,\,0\, \bigr)$.
The coordinates are the shuffle equations.
We saw these in (\ref{eq:42shuffle}).
\hfill $\diamond$ \end{example}

In the next example we demonstrate how the shuffle equations
can be computed and analyzed within the computer algebra
system {\tt Macaulay2} \cite{M2}. All computations of the varieties
${\rm Gr}(\ell,n)^T$ in this paper were carried out by
this code, with ${\tt n}, {\tt l}$ and ${\tt U} = {\rm Id}_n +z T$ adjusted.

\begin{example} \label{ex:M2}
We examine the smallest instance where 
${\rm Gr}(\ell,n)^T$ has three irreducible components,
namely
 $n{=}6, \ell {=} 3$ and $\lambda = (3,1,1,1)$,
as seen in Table \ref{tab:67} below.
The following {\tt Macaulay2} code outputs the ideal
{\tt J}  generated by the shuffle equations and Pl\"ucker~quadrics:

\begin{verbatim}
n=6; l=3;
R = ring Grassmannian(l-1,n-1,CoefficientRing => QQ);
P = matrix{gens R}; S = R[z];
U = matrix {{1,z,0,0,0,0},
            {0,1,z,0,0,0},
            {0,0,1,0,0,0},
            {0,0,0,1,0,0},
            {0,0,0,0,1,0},
            {0,0,0,0,0,1}};

M = (toList coefficients(P*exteriorPower(l,U)))_1;
rowws = toList(0..((# entries M)-2));
I = minors(1,submatrix(M,rowws,))
J = I+Grassmannian(l-1,n-1,R); toString mingens J
betti mingens J, (dim J)-2, degree J 
J == radical(J), isPrime J
\end{verbatim}

The output of this code shows that the ideal ${\tt J}$
is radical but not prime. It is
 minimally generated by $12$ linear forms and $8$ quadrics.
Its variety ${\rm Gr}(3,6)^T$ has dimension $4$ and degree~$2$ in 
the ambient space $ \PP^{19}$ of ${\rm Gr}(3,6)$.
We next compute the prime decomposition:
\begin{verbatim}
DJ =  decompose J; #DJ, betti mingens radical J
apply(DJ, T -> {T,codim T, degree T, betti mingens T})
\end{verbatim}
The fixed point locus ${\rm Gr}(3,6)^T$ has three irreducible
components. The largest component is defined by a quadric in a subspace $\PP^5$.
In addition, there are
two coordinate subspaces $\PP^3$.
\hfill $\diamond$ \end{example}

We next come to a duality result which will aid our computations in Section \ref{sec3}.

\begin{theorem}\label{th:duality}
The varieties ${\rm Gr}(\ell,n)^T$ and ${\rm Gr}(n-\ell,n)^T$ coincide
after a linear change of coordinates in the ambient space $\,\PP^{\binom{n}{\ell}-1}$. This holds
for all $\,\ell,n$ and all nilpotent $n \times n$ matrices $T$.
Under this coordinate change, which depends on $T$, 
the shuffle equations coincide.
\end{theorem}

\begin{proof}
Let  ${\rm B}_m = (b_{ij})$ denote the $m\times m$  matrix with $1$'s on the antidiagonal and $0$'s elsewhere, 
i.e.~$b_{ij} = 1$ if $i+j=m+1$ and $b_{ij} = 0$ otherwise.
Given a partition $\lambda$ of $n$ and its
matrix $T=T_{\lambda}$, we define $B=B_{\lambda}$ to be the block-diagonal
$n \times n$ matrix $B_{\lambda}={\rm diag}({\rm B}_{\lambda_1},\ldots, {\rm B}_{\lambda_s})$. 
Note that $B^2 = {\rm Id}_n$ and, if $T^{\rm t}$ denotes the transpose $T$,  that $TB=BT^{\rm t}$, i.e.\ $T$ is a self-adjoint linear operator
for the non-degenerate symmetric bilinear form defined by $B$ on $K^n$. 

We consider the non-degenerate inner product on $K^n$ that is 
defined by the invertible symmetric matrix $B$. The orthogonal space of
a given $\ell$-dimensional subspace $L$ with respect to this inner product is
the $(n-\ell)$-dimensional subspace 
$$ L^\perp \,\, = \,\, {\rm ker} ({\bf L} B) \,\, = \,\,
\bigl\{ \,v \in K^n \,: \, u B v^t = 0 \, \,\hbox{for all}\, \, u \in L \bigr\}. $$
Suppose $L$ is $T$-fixed. We claim that $L^\perp$ is $T$-fixed.
 Indeed, suppose $v \in L^\perp$, i.e.~$uB v^t  = 0$ for all $u \in L$. This implies
$uB (vT)^t = uBT^t v^t =  (uT)B  v^t = 0$ for all $u \in L$, and so
$vT \in L^\perp$.
This shows that passing to the orthogonal space defines the desired linear isomorphism
\begin{equation}
\label{equ:isograss}
{\rm Gr}(\ell,n)^T\longrightarrow {\rm Gr}(n-\ell,n)^T, \quad L\longmapsto L^{\perp}.
\end{equation}
For $T = {\rm 0}_n$, this is the familiar isomorphism between the Grassmannians
${\rm Gr}(\ell,n)$ and ${\rm Gr}(n-\ell,n)$. A subtle point is that
duality is taken relative to the inner product given by~$B$.

We shall explicitly describe the linear change of coordinates on $\PP^{\binom{n}{\ell}-1}$
that induces the isomorphism (\ref{equ:isograss}).
We start with the {\em Hodge star isomorphism}
$P \mapsto P^*$ that takes the vector
$P = (p_{i_1  \cdots i_\ell})_{1 \leq i_1  < \cdots < i_\ell \leq n}$
to the vector $P^* = (p^*_{j_1 \cdots j_{n-\ell}})_{1 \leq j_1 <  \cdots < j_{n-\ell} \leq n}$.
If $I$ is an ordered $\ell$-subset of $[n] = \{1,2,\ldots,n\}$ and
$J = [n] \backslash I$ is the complementary ordered $(n-\ell)$-subset then 
$$ p^*_J \,\, := \,\, {\rm sign}(I,J) \cdot p_I. $$
Here ${\rm sign}(I,J)$ is the sign of the permutation of $[n]$
given by the ordered sequence $(I,J)$.

To be completely explicit, here is an example. For $n{=}4$ the formula for the Hodge star~is 
\begin{equation}
\label{eq:hodgestar}  \begin{matrix}
P^* & = & ( \,p_1^*,p_2^*,p_3^*,p_4^*\, \mid \,p_{12}^*,p_{13}^*,
p_{23}^*,p_{14}^*,p_{24}^*,p_{34}^*\,\mid\,p_{123}^*,p_{124}^*,p_{134}^*,p_{234}^*\, ) \\  & = &
( p_{234}, -p_{134}, p_{124}, -p_{123}\, \mid \, p_{34}, -p_{24}, p_{14}, 
p_{23}, -p_{13}, p_{12}\,\mid \, p_4, -p_3, p_2, -p_1 ).
\end{matrix}
\end{equation}
The restriction of the Hodge star to the Grassmannian ${\rm Gr}(\ell,n)$ in $\PP^{\binom{n}{\ell}-1}$
takes a linear space to its orthogonal space with respect to the standard inner product.
To incorporate the quadratic form $B$, we consider the  automorphism of $\PP^{\binom{n}{\ell}-1}$ 
that takes $P$ to $\bigl(P \cdot (\wedge_\ell B)\bigr)^*$. The restriction of this automorphism to
the Grassmannian ${\rm Gr}(\ell,n)$ is the isomorphism (\ref{equ:isograss}). 

It remains to show that the map $P \mapsto \bigl(P \cdot (\wedge_\ell B)\bigr)^*$
preserves the shuffle equations. 
To~do this, let $M_{\ell}$ be the ${\binom{n}{\ell}\times \binom{n}{\ell}}$-matrix with entries in $K$ such that $P^*=PM_{\ell}$. Note that $M_{\ell}^2$ is the identity matrix.   Via conjugation, the Hodge star operator extends to $ K^{\binom{n}{\ell}\times \binom{n}{\ell}}$ i.e.\ via sending $N$ to $N^*=M_{\ell}NM_{\ell}$.  Assuming that $P\cdot\wedge_\ell(\mathrm{Id}_n+zT)=P$ for all $z\in K$,  we rewrite
\begin{align*}
(P\cdot(\wedge_\ell B))^*\wedge_{n-\ell} ({\rm Id}_n + z T)& =P\cdot(\wedge_\ell B)M_\ell\wedge_{n-\ell} ({\rm Id}_n + z T)
\\
&=P\cdot(\wedge_\ell B)\wedge_{\ell} ({\rm Id}_n + z T^t)(\wedge_{\ell}B)(\wedge_{\ell}B)M_{\ell}\\
& = P\cdot\wedge_{\ell} ({\rm Id}_n + z BT^tB)(\wedge_{\ell}B)M_{\ell}\\
& = P\cdot\wedge_{\ell} ({\rm Id}_n + z T)(\wedge_{\ell}B)M_{\ell}\\
&=P\cdot(\wedge_{\ell}B)M_{\ell}\\
&=(P\cdot(\wedge_\ell B))^*.
\end{align*}
This shows
that the shuffle equations for ${\rm Gr}(\ell,n)^T$ are mapped to those of ${\rm Gr}(n-\ell,n)^T$ under our automorphism of $\PP^{\binom{n}{\ell}-1}$. This was the claim, and the proof  of Theorem \ref{th:duality} is complete.
\end{proof}

\begin{example}[$n=4$] 
Fix $\epsilon = 0$ in \cref{ex:42}.
Then $T=T_{\lambda}$ for $\lambda=(2,1,1)$, and we have
\[
B = 
B_{\lambda}\,= \,\begin{small} \begin{pmatrix}
0 & 1 & 0 & 0 \\
1 & 0 & 0 & 0 \\
0 & 0 & 1 & 0 \\
0 & 0 & 0 & 1
\end{pmatrix}. \end{small}
\]
Our map $P \mapsto \bigl(P \cdot (\wedge_\ell B)\bigr)^*$, written  for $\ell=1,2,3$, is
the following signed permutation of~(\ref{eq:hodgestar}):
\begin{equation}\label{eq:hodgestarB} P \,\,\, \mapsto \,\,\,
(\,-p_{134}, p_{234}, p_{124}, -p_{123}\,\mid \, -p_{34}, p_{14}, -p_{24}, -p_{13}, p_{23}, p_{12}\,\mid \, -p_4, p_3, -p_1, p_2\,) .
\end{equation}
For $\ell = 1$ the unique shuffle equation is $p_1$.
This is mapped to $-p_{134}$, which is the unique shuffle equation for $\ell=3$.
Likewise, $p_{134}$ is mapped to $-p_1$. This makes sense because
${\rm Gr}(1,4)^T = V(p_1) = {\rm span}(e_2,e_3,e_4) = {\rm ker}(T)$,
whereas ${\rm Gr}(3,4)^T = V(p_{134})$ consists of all hyperplanes in $K^4$ that contain $e_2$.
Both are projective planes $\PP^2$. Our involution swaps~them.

For $\ell = n-\ell = 2$, there are two shuffle equations, namely $p_{13}$ and $p_{14}$,
as seen in Example~\ref{ex:drei}. These two Pl\"ucker coordinates are swapped
(up to sign) in (\ref{eq:hodgestarB}), so our involution fixes ${\rm Gr}(2,4)^T$.
Moreover, this involution  interchanges the two irreducible components in (\ref{eq:42shuffledeco}).
We see this in the coordinate change (\ref{eq:hodgestarB}) which sends
$p_{12} \mapsto - p_{34}$  and $p_{34} \mapsto p_{12}$.
\hfill $\diamond$ \end{example}

\section{Classification and Counterexample}
\label{sec3}

The main result in this article is the determination of
all fixed point loci ${\rm Gr}(\ell,n)^T$ for $n \leq 8$. 
From this computational result, we extract the following theorem about the shuffle equations.

\begin{theorem}
Fix $1 \leq \ell < n \leq 7$ and let $T$ be any nilpotent $n \times n$ matrix.
Then the shuffle equations generate the radical ideal of the
fixed point locus ${\rm Gr}(\ell,n)^T$.
The same does not hold for $n=8$: there
 is a unique partition, namely $\lambda = (4,2,2)$,
and a unique dimension, namely $\ell = 4$, such that
the radical ideal of $\,{\rm Gr}(\ell,n)^{T_\lambda}$ is not generated by the shuffle equations.
\end{theorem}

\begin{proof}
The proof is carried out by exhaustive computation of all
varieties ${\rm Gr}(\ell,n)^{T_\lambda}$ where
$\lambda$ is any partition of $n \leq 8$. Here we use the 
{\tt Macaulay2} code from Example \ref{ex:M2} and \cref{th:duality}.

The results are summarized in Tables \ref{tab:45}, \ref{tab:67} and \ref{tab:8}.
For each instance $(\lambda,\ell)$, we report a
triple $[\sigma,\delta,\gamma]$ or
$[\sigma,\delta,\gamma]^\kappa$. Here
$\sigma $ is the number of linearly independent
shuffle equations. The entries $\delta$ and $\gamma$
are the dimension and degree of ${\rm Gr}(\ell,n)^{T_\lambda}$
in its Pl\"ucker embedding into $\PP^{\binom{n}{\ell}-1}$.
The upper index $\kappa$ is the number of irreducible components of ${\rm Gr}(\ell,n)^T$,
and this index is dropped if $\kappa=1$. The columns for $\ell > n/2$ are omitted because
of Theorem \ref{th:duality}. In any given row of one of our tables, the entry for $n-\ell$
would be identical to that for $\ell$.

\begin{table}[h]
\begin{center}
\begin{tabular}{||c c c||} 
 \hline
 $\lambda$ & $\ell=1$ & $\ell=2$ \\ [0.3ex] 
 \hline\hline
  (1,1,1,1) &  [0,3,1] &  [0,4,2]     \\  \hline
  (2,1,1)   &  [1,2,1] &  $\,$[2,2,2]${}^2$  \\  \hline
  (2,2)     &  [2,1,1] &  [2,2,2]     \\  \hline
  (3,1)     &  [2,1,1] &  [4,1,1]     \\  \hline
 (4)       &  [3,0,1] &  [5,0,1]     \\  [0.3ex]
    \hline
\end{tabular} \qquad \qquad \begin{small}
\begin{tabular}{||c c c||} 
 \hline
 $\lambda$ & $\ell=1$ & $\ell=2$ \\ [0.3ex] 
 \hline\hline
  (1,1,1,1,1) &  [0,4,1] & [0,6,5]    \\  \hline
  (2,1,1,1)   &  [1,3,1] & $\,$[3,4,2]${}^2$ \\  \hline
  (2,2,1)     &  [2,2,1] & [4,3,3]     \\  \hline
  (3,1,1)     &  [2,2,1] & $\,$[6,2,2]${}^2$ \\  \hline
  (3,2)       &  [3,1,1] & [6,2,2]    \\  \hline
  (4,1)       &  [3,1,1] & [8,1,1]    \\  \hline
  (5)         &  [4,0,1] & [9,0,1]    \\  [0.3ex]
  \hline
  \end{tabular} \end{small}
\caption{\label{tab:45} Fixed point loci $\,{\rm Gr}(\ell,n)^{T}$ for $n=4$ and $n=5$.}
\end{center}
\end{table}

\begin{table}[h]
\begin{center}
\begin{small}
\begin{tabular}{||c c c c||} 
\hline
 $\lambda$ & $\ell=1$ & $\ell=2$ & $\ell = 3$ \\ [0.3ex] 
 \hline\hline
 $\!$(1,1,1,1,1,1)$\!$ & [0,5,1] & [0,8,14]    & [0,9,42]        \\  \hline
   (2,1,1,1,1) & [1,4,1] & [4,6,5]${}^2$  & [6,6,10]${}^2$     \\  \hline
     (2,2,1,1) & [2,3,1] & [6,4,6]${}^2$  & [8,5,10]        \\  \hline
     (3,1,1,1) & [2,3,1] & [8,4,2]${}^2$  & [12,4,2]${}^3$     \\  \hline
       (2,2,2) & [3,2,1] & [6,4,6]    & [11,4,6]        \\  \hline
       (3,2,1)  & [3,2,1] & [9,3,3]     & [12,3,6]${}^2$     \\  \hline
       (4,1,1) & [3,2,1] & [11,2,2]${}^2$ & [16,2,2]${}^2$     \\  \hline
         (3,3) & [4,1,1] & [11,2,2]    & [12,3,6]        \\  \hline
         (4,2) & [4,1,1] & [11,2,2]    & [16,2,2]        \\  \hline
	 (5,1) & [4,1,1] & [13,1,1]    & [18,1,1]        \\  \hline
	   (6) & [5,0,1] & [14,0,1]    & [19,0,1]        \\  [0.3ex]
 \hline
  \end{tabular} $\,\,$
\begin{tabular}{||c c c c||} 
\hline
 $\lambda$ & $\ell=1$ & $\ell=2$ & $\ell = 3$ \\ [0.3ex]   
  \hline\hline
$\!\!$(1,1,1,1,1,1,1)$\!\!$ & [0,6,1] & [35,10,42]  &$\!\!$ [140,12,462] $\!\!$   \\  \hline
  (2,1,1,1,1,1) & [1,5,1] & [5,8,14]${}^2$ & [10,9,42]${}^2$    \\  \hline
    (2,2,1,1,1) & [2,4,1] & [8,6,5]${}^2$  & [14,7,35]${}^2$    \\  \hline
    (3,1,1,1,1) & [2,4,1] & [10,6,5]${}^2$ & [20,6,10]${}^3$    \\  \hline
      (2,2,2,1) & [3,3,1] & [9,5,10]    & [17,6,30]       \\  \hline
      (3,2,1,1) & [3,3,1] & [12,4,6]${}^2$ & [21,5,10]${}^2$    \\  \hline
      (4,1,1,1) & [3,3,1] & [14,4,2]${}^2$ & [27,4,2]${}^3$     \\  \hline
        (3,2,2) & [4,2,1] & [12,4,6]    & [23,4,12]${}^2$    \\  \hline
        (3,3,1) & [4,2,1] & [15,3,3]    & [23,4,12]       \\  \hline
        (4,2,1) & [4,2,1] & [15,3,3]    & [27,3,6]${}^2$     \\  \hline
        (5,1,1) & [4,2,1] & [17,2,2]${}^2$ & [31,2,2]${}^2$     \\  \hline
          (4,3) & [5,1,1] & [17,2,2]    & [27,3,6]        \\  \hline
          (5,2) & [5,1,1] & [17,2,2]    & [31,2,2]        \\  \hline
          (6,1) & [5,1,1] & [19,1,1]    & [33,1,1]        \\  \hline
            (7) & [6,0,1] & [20,0,1]    & [34,0,1]        \\  [0.3ex]
 \hline
  \end{tabular}
  \end{small}
      \caption{\label{tab:67} Fixed point loci $\,{\rm Gr}(\ell,n)^{T}$ for $n=6$ and $n=7$.}
      \end{center}
\end{table}	   

In each case, we computed the irreducible components of the shuffle ideal.
We recorded the prime ideal for each component, and we determined
degree, dimension, singularities etc. The intersection of these primes is
the radical ideal of ${\rm Gr}(\ell,n)^T$. In all cases but one,
we found that the radical ideal is generated by the shuffle equations
plus the Pl\"ucker quadrics. The unique exceptional case is
 $\lambda = (4,2,2)$ and $\ell = 4$, with the highlighted
entry    {\bf [54,4,24]}${}^{\bf 3}$. 
This means that
 there are $54$ linearly independent
shuffle relations plus $4$ additional Pl\"ucker quadrics. However,
this ideal is not radical. To generate the radical, we need one more linear form.
Further below, we shall examine the geometry of this counterexample in detail.

An easy {\tt Macaulay2} proof for the failure of {\tt J} to be radical is running the following line:
\begin{verbatim}
        apply(first entries promote(P,S),p -> {p % J, p^2 % J})
\end{verbatim}
This reveals that the variable $p_{1468}$ is not in {\tt J} but its
square is in {\tt J}. Note that this coordinate corresponds to $ {\tt p}_{{\tt 0357}}$
in the
  zero-based indexing of {\tt Macaulay2}. This concludes the proof. \end{proof}

\begin{example}[$n=6$] \label{ex:(6)}
Consider the lower right entry on the left in Table \ref{tab:67}.
Here $\ell=3$ and $\lambda = (6)$, so $T$
is the nilpotent matrix that maps $e_1 \mapsto e_2 \mapsto
e_3 \mapsto e_4 \mapsto e_5 \mapsto e_6 \mapsto 0$.
The variety ${\rm Gr}(3,6)^T$ consists of a single point $e_{456}$.
It is instructive to revisit the construction of the shuffle equations
for this case. The $20$ coordinates of the row vector 
$P \cdot  \wedge_3 ({\rm Id}_6 + z T) $~are
$$ \begin{matrix}
p_{123}, \,
p_{123} z + p_{124},\,
p_{123} z^2 + p_{124} z + p_{134},\,
p_{123} z^3 + p_{124} z^2 + p_{134} z + p_{234},\,
p_{124} z + p_{125},\,  \\
p_{124} z^2 + (p_{134} + p_{125}) z + p_{135} ,\,
p_{124} z^3 + (p_{134} + p_{125}) z^2 + (p_{234} + p_{135})z + p_{235},\,
p_{134} z^2 + p_{135} z \\ +  p_{145} ,\, 
p_{134} z^3 + (p_{234} + p_{135}) z^2 + (p_{235} + p_{145}) z + p_{245},\, 
p_{234} z^3 + p_{235} z^2 + p_{245} z + p_{345} ,\\
p_{125} z + p_{126},\,
p_{125} z^2 + (p_{135} + p_{126}) z + p_{136},\,
p_{125} z^3 + (p_{135} + p_{126}) z^2 + (p_{235} + p_{136}) z + p_{236},\\
p_{135} z^2 + (p_{145} + p_{136}) z + p_{146},\,
p_{135} z^3 + (p_{235} + p_{145} + p_{136}) z^2 + (p_{245} + p_{236} + p_{146}) z + p_{246},\\
p_{235} z^3 + p_{245} + p_{236} z^2 + (p_{345} + p_{246})z + p_{346},\,
p_{145} z^2 + p_{146} z + p_{156},\, 
p_{145} z^3 + (p_{245} + p_{146}) z^2 + \\  (p_{246} {+} p_{156})z + p_{256},\,
p_{245} z^3 + (p_{345} {+} p_{246}) z^2 + (p_{346} {+} p_{256})z + p_{356},\,
p_{345}z^3 {+} p_{346} z^2 {+} p_{356} z {+} p_{456}.
\end{matrix}
$$
The shuffle equations are the coefficients of $z^3,z^2$ and $z$. They span  the ideal of
all Pl\"ucker coordinates except $p_{456}$. This is the homogeneous maximal
ideal of $ \,         {\rm Gr}(3,6)^T = \{e_{456}\}$. \hfill $\diamond$
\end{example}
   
\begin{example}[$n{=}8$]
The smallest instance of a variety ${\rm Gr}(\ell,n)^{T_\lambda}$
with four irreducible components occurs for
$n=8$, $\lambda = (4,1,1,1,1)$, and $\ell = 4$.
There are $54$ linearly independent shuffle equations,
and $46$ Pl\"ucker quadrics remain modulo these linear forms.
The variety ${\rm Gr}(4,8)^{T_\lambda}$ has dimension
$6$ and degree $10$ in $\PP^{69}$. It is the union of four irreducible components,
two of dimension $6$ and degree $5$,
and two linear spaces of dimension $4$. 
\hfill $\diamond$ \end{example}

\begin{table}[h]
\begin{center}
\begin{tabular}{||c c c c||} 
\hline
 $\lambda$ & $\ell=2$ & $\ell=3$ & $\ell = 4$ \\ [0.3ex]   
  \hline\hline
$\!\!$ (1,1,1,1,1,1,1,1)$\!\!$ &  [0,12,132]  &  $\!$[420,15,6006] $\!$ & $\!\!\!$ [721,16,24024]  $\!\!$   \\  \hline
  (2,1,1,1,1,1,1) &  [6,10,42]${}^2$ & [15,12,462]${}^2$ &  [20,12,924]${}^2$   \\  \hline
    (2,2,1,1,1,1) &  [10,8,14]${}^2$ &  [22,9,168]${}^2$ & [28,10,420]${}^3$   \\  \hline
    (3,1,1,1,1,1) &  [12,8,14]${}^2$ &  [30,9,42]${}^3$  &  [40,9,42]${}^3$   \\  \hline
      (2,2,2,1,1) &  [12,6,20]${}^2$ &  [26,8,140]    &  [34,8,280]${}^2$   \\  \hline
      (3,2,1,1,1) &  [15,6,5]${}^2$  &  [33,7,35]${}^3$  &  [42,7,70]${}^2$   \\  \hline
      (4,1,1,1,1) &  [17,6,5]${}^2$  &  [41,6,10]${}^3$  & [54,6,10]${}^4$    \\  \hline
        (2,2,2,2) &  [12,6,20]    &  [32,7,70]     &  [34,8,280]   \\  \hline
	(3,2,2,1) &  [16,5,10]    &  [35,6,30]${}^2$  &  [46,6,60]${}^2$    \\  \hline
        (3,3,1,1) &  [19,4,6]${}^2$  &  [38,5,30]${}^2$  &  [46,6,60]    \\  \hline
        (4,2,1,1) &  [19,4,6]${}^2$  &  [42,5,10]${}^2$  &  [54,5,10]${}^3$    \\  \hline
        (5,1,1,1) &  [21,4,2]${}^2$  &  [48,4,2]${}^3$   &  [62,4,2]${}^3$    \\  \hline
          (3,3,2) &  [19,4,6]    &  [38,5,30]    &  [52,5,30]    \\  \hline
          (4,2,2) &  [19,4,6]     &  [44,4,12]${}^2$  &  {\bf [54,4,24]}${}^{\bf 3}$  \\  \hline
          (4,3,1) &  [22,3,3]     &  [44,4,12]     &  [54,4,24]${}^2$    \\  \hline
          (5,2,1) &  [22,3,3]     &  [48,3,6]${}^2$   &  [62,3,6]${}^2$    \\  \hline
          (6,1,1) &  [24,2,2]${}^2$  & [52,2,2]${}^2$   &  [66,2,2]${}^2$    \\  \hline
            (4,4) &  [24,2,2]    &  [48,3,6]   &  [54,4,24]    \\  \hline
            (5,3) &  [24,2,2]     &  [48,3,6]      &  [62,3,6]    \\  \hline
	    (6,2) &  [24,2,2]     &  [52,2,2]      &  [66,2,2]    \\  \hline
	    (7,1) &  [26,1,1]     &  [54,1,1]      &  [68,1,1]    \\  \hline
	      (8) &  [27,0,1]     &  [55,0,1]     &  [69,0,1]    \\  \hline
      \hline
  \end{tabular}
  \caption{\label{tab:8} Fixed point loci $\,{\rm Gr}(\ell,n)^{T}$ for $n=8$.}
  \end{center}
\end{table}

We now present a detailed study of our counterexample to
\cite[Conjecture 7.6]{MuWeYa}. We have
$n=8$, $\ell=4$, and the matrix $T = T_\lambda$ given by the partition $\lambda=(4,2,2)$, i.e. operating as
\[
e_1\mapsto e_2\mapsto e_3\mapsto e_4 \mapsto 0, \quad e_5\mapsto e_6 \mapsto 0,
 \quad e_7 \mapsto e_8 \mapsto 0.
\]
We consider the scheme structure on
${\rm Gr}(4,8)^T$ given by the shuffle ideal
{\tt J}. There are
three minimal primes,
each of dimension $4$ and degree $6$. One component is non-reduced of
multiplicity $2$, so the degree of our scheme is $24 = 6+6+2\cdot 6$.
It has no embedded primes.

We begin with the two reduced components. Each of these is a Segre
fourfold $\PP^2 \times \PP^2$ lying in a $\PP^8$ inside a coordinate subspace $\PP^{11}$.
The two ambient coordinate subspaces are
$$ \begin{matrix} 
  {\rm span}\{e_{1234},e_{1346},e_{1348},e_{2345},e_{2346},e_{2347},e_{2348},
   e_{3456},e_{3458},e_{3467},e_{3468},e_{3478}\} ,
  \\
  {\rm span} \{e_{3456},e_{3458},e_{3467},e_{3468},e_{3478},e_{3568},
  e_{3678},e_{4567},e_{4568},e_{4578},e_{4678},e_{5678}\}.
  \end{matrix}
  $$
In suitable affine coordinates, the two reduced components are parametrized by
$$
\begin{small}
\begin{pmatrix}
1 & 0 &  0 &  0 &  a & b &  c & d \\ 
0 & 1 &  0 & 0 &  0 & a &  0 &  c\\
0 &  0 &  1 &  0 & 0 &  0 &   0 &  0\\
0 &  0 & 0 & 1 & 0 & 0 & 0 & 0 
\end{pmatrix}
\end{small} \qquad
{\rm and} \qquad
\begin{small}
\begin{pmatrix}
0 & 0 &  a & b &  1 & 0 &  0 & 0\\
0 & 0 & 0 & a & 0 & 1 &  0 & 0\\
0 & 0 & c & d & 0 & 0 & 1 & 0\\
0 & 0 & 0 & c & 0 & 0 & 0 & 1
\end{pmatrix}.
\end{small}
$$
The $T$-module structures on these
subspaces $L$ are given by the partitions $(4)$
and $(2,2)$.

We now study the non-reduced component. It lies in a $\PP^8$ inside the coordinate subspace
$$ {\rm span} \{e_{3468}, e_{2346},e_{2348},e_{2468},e_{3456},e_{3458},e_{3467},e_{3478},e_{4568},e_{4678}\}
 \,\, \simeq \,\,
\mathbb{P}^9.
$$
Geometrically, it is a 
cone over a hyperplane slice of $\mathbb{P}^2 \times \mathbb{P}^2$.
It has the matrix representation 
$$ \begin{small}
\begin{pmatrix}
0 & a & 1 & 0 & b & 0 & c & 0 \\
0 & 0 & 0 & 1 &  0 & 0 & 0 & 0 \\
0 & d & 0 &  0 &  e & 1 &  f & 0\\
0 & g & 0 & 0 &  h & 0 &  i & 1
\end{pmatrix}
\end{small}
\,\,\hbox{where the $3 \times 3$ block}\,
\begin{small}
\begin{pmatrix}
a & \! b & \! c \\ d & \! e & \! f \\ g & \! h & \! i
\end{pmatrix} \end{small}
\hbox{
has trace $0$ and rank $\leq 1$}.
$$
The zero matrix gives a singular point on this component.
There are moreover three distinct $T$-module structures on the
subspaces $L$ in this component, namely $(3,1)$, $(2,2)$ and $(2,1,1)$.

\begin{remark} The first nontrivial entry in each table 
is $\lambda = (2,1,1,\ldots,1)$. Each irreducible component of 
${\rm Gr}(\ell,n)^{T_\lambda}$ is a Grassmannian. 
This is obvious for $\ell = 1$. We sketch a proof for $2 \leq \ell \leq\lceil n/2\rceil $.
The matrix $T$ maps
$e_1 \mapsto e_2$ and $e_i \mapsto 0$ for $i \geq 2$. 
Thus, $\ker (T)$ is a hyperplane in $K^n$ and each $L\in{\rm Gr}(\ell,n)^T$ possesses a minimal subspace $\tilde{L}$ 
satisfying $L=\tilde{L}+\tilde{L}T$. The space $\tilde{L}$ 
might not be unique, but its dimension is, being either $\ell$ or $\ell-1$. 
In the first case, $\tilde{L}=L$ and $L$ is a subspace of $\ker (T)$. 
In the second case, $\tilde{L}$ is~an $(\ell-1)$-dimensional
 subspace of $K^n$ with $\tilde{L}\not\subseteq \ker T$ and $\tilde{L}T={\rm span}(e_2)$. 
 This implies that  ${\rm Gr}(\ell,n)^T$ has two irreducible components, namely the Grassmannians
  ${\rm Gr}(\ell,n-1)$ and ${\rm Gr}(\ell-1,n)$.
\end{remark}

\section{The Affine Grassmannian}
\label{sec4}

The key player in the articles \cite{KLMW} and \cite{MuWeYa}
is the  affine Grassmannian, which is an infinite-dimensional variety.
Our varieties ${\rm Gr}(\ell,n)^T$ serve as finite-dimensional models, 
when restricting to
$T= T_\lambda$ where $\lambda$ is a 
{\em rectangular partition}.
By this we mean partitions $\lambda = (r,r,\ldots,r) $ with $d$
parts, so that $dr = n$ and $d,r \geq 2$.
This section revolves around the next two points.

\begin{theorem} \label{thm:rec}
If $\lambda$ is a rectangular partition then 
the variety ${\rm Gr}(\ell,n)^{T_\lambda}$ is irreducible.
\end{theorem}

\begin{conjecture} \label{conj:rec}
 Conjecture 7.6 in \cite{MuWeYa} holds for
rectangular partitions $\lambda$. In other words, for rectangular partitions,
the shuffle equations plus Pl\"ucker quadrics generate a prime ideal.
\end{conjecture}

\begin{remark}
Tables \ref{tab:45}, \ref{tab:67} and \ref{tab:8} show
that Theorem \ref{thm:rec} and Conjecture \ref{conj:rec}
are true for $n \leq 8$. In that range,
the only rectangular partitions $\lambda$ are
$(2,2), (2,2,2), (3,3), (2,2,2,2)$ and $ (4,4)$.
We see that the shuffle ideals that cut out their varieties
${\rm Gr}(\ell,n)^{T_\lambda}$ are prime for all $\ell$.
\end{remark}

We shall derive Theorem \ref{thm:rec} from known facts about
Schubert varieties in affine Grassmannians. We aim to explain this approach in
a manner that is as self-contained as possible.
The section concludes with some
further  evidence 
in support of Conjecture \ref{conj:rec}.

Let $\KK = K(\!(t)\!)$ be the field of Laurent series with coefficients in K.
Its valuation ring $\mathcal{O}_\KK = K[[t]]$ consists of  formal
power series with nonnegative integer exponents. The residue field
is $K$. The $\KK$-vector space $\KK^d$
is a module over $\mathcal{O}_\KK$. A {\em lattice} $L$ is
an $\mathcal{O}_\KK$-submodule of $\KK^d$ of maximal rank $d$.
Two lattices $L$ and $L'$ are {\em equivalent} if
$L' = t^a L$ for some $a \in \ZZ$.
To parametrize all lattices,
 we consider the groups
${\rm GL}_d(\KK)$ and ${\rm GL}_d(\mathcal{O}_\KK)$
of invertible $d \times d$ matrices, with entries
in $\KK$ and $\mathcal{O}_\KK$ respectively. The {\em affine Grassmannian} is the coset space
\begin{equation}
\label{eq:affgrass}  {\rm GL}_d(\KK)\,/\,{\rm GL}_d(\mathcal{O}_\KK). 
\end{equation}
Its points are the lattices $L$.
Indeed, every $L$ is the column span over $\mathcal{O}_\KK$ 
of a matrix in ${\rm GL}_d(\KK)$. Two matrices define
the same $L$ if they differ via right multiplication 
by a matrix in ${\rm GL}_d(\mathcal{O}_\KK)$.
  To obtain finite-dimensional varieties we can study with
a computer, we set
\begin{equation}
\label{eq:ballr}
 \mathcal{B}_r \,\, = \,\,  \{ \, L \,\, {\rm lattice} \,\,: \,\, t^{r} \mathcal{O}_\KK^d \subseteq L \subseteq \mathcal{O}_\KK^d\, \} . \end{equation}
We note that (\ref{eq:affgrass}) modulo equivalence of lattices equals
the  {\em Bruhat-Tits building} for ${\rm GL}_d(\KK)$. The set
    $\mathcal{B}_r$ represents the ball of radius $r$ around
   the standard lattice $\mathcal{O}_\KK^d$ in that building.

Both $ t^{r} \mathcal{O}_\KK^d $ and  $\mathcal{O}_\KK^d $
are infinite-dimensional vector spaces over $K$. Their quotient
is a finite-dimensional vector space over $K$. This space has dimension $n=dr$ and we identify
\begin{equation}
\label{eq:identification}  K^n \,\, = \,\, \mathcal{O}_\KK^d \,/\, t^{r} \mathcal{O}_\KK^d .
\end{equation}
Writing $e_1,e_2,\ldots,e_d$ for the standard basis of $K^d$, we shall use
 the following basis for~$K^n$:
 \begin{equation}
 \label{eq:nicebasis}
 e_1,te_1,\ldots,t^{r-1}e_1, \,e_2,t e_2,\ldots,t^{r-1}e_2,\,\ldots\ldots,\,
 e_d, t e_d,\ldots, t^{r-1} e_d. 
 \end{equation}
In this basis, 
 multiplication with $t$ is given by the nilpotent $n \times n$ matrix $T_\lambda$ for $\lambda=(r,\dots,r)$.

Every lattice $L \in \mathcal{B}_r$ is determined by its
image in (\ref{eq:identification}).
We also write $L$ for that image.
Hence $ L$ is a  subspace of $K^n$ that satisfies $LT_\lambda \subseteq L$.
Conversely, every subspace $L$ of $K^n$ satisfying
$LT_\lambda \subseteq L$ comes from a unique lattice in $\mathcal{B}_r$.
This establishes the following result.

\begin{proposition}
The radius $r$ ball in \emph{(\ref{eq:ballr})} is the following finite union of projective varieties:
\begin{equation} \label{eq:Brunion}
\qquad \mathcal{B}_r \,\,\,=\,\, \,
    \bigcup_{\ell=0}^{dr}\,  \,\mathrm{Gr}(\ell,n)^{T_\lambda},
    \qquad \hbox{where}\,\,\, \lambda \,=\, (r,r,\ldots,r).
\end{equation}
\end{proposition}

\begin{example}[$d=r=2$]
Here $n=rd=4$, $T = T_{(2,2)}$, and the disjoint union in (\ref{eq:Brunion}) equals
$$ \mathcal{B}_2 \,\,\,=\,\, \,
\mathrm{Gr}(0,4)^T \, \cup \,
\mathrm{Gr}(1,4)^T \, \cup \,
\mathrm{Gr}(2,4)^T \, \cup \,
\mathrm{Gr}(3,4)^T \, \cup \,
\mathrm{Gr}(4,4)^T .
$$
The first and last Grassmannian are the points that represent the
 lattices $\mathcal{O}_\KK^2$
and $t^2\mathcal{O}_\KK^2$.
The second and fourth Grassmannian are projective lines $\PP^1$.
The middle Grassmannian is a quadratic cone in $\PP^3$.
We saw this in (\ref{eq:42shuffle}) for $\epsilon = 1$.
Note the row $\lambda = (2,2)$ in Table  \ref{tab:45}.
\hfill $\diamond$
\end{example}

\begin{example}[$n=8$]
The two options are $d=4, r=2$ and $d=2,r=4$.
These are the rows
$\lambda = (2,2,2,2)$ and
$\lambda = (4,4)$ of Table \ref{tab:8}.
In either case, $\mathcal{B}_r$ is the disjoint union 
of nine irreducible varieties, indexed by
$\ell = 0,1,2,3,4,5,6,7,8$, and soon to be called Schubert varieties.
Their dimensions are $0,3,6,7,8,7,6, 3,0$
and $0,1,2,3,4,3,2,1,0$ respectively.
\hfill $\diamond$
 \end{example}
 
 We now turn towards the proof of Theorem \ref{thm:rec}.
 We shall give a polynomial parametrization for each variety
   ${\rm Gr}(\ell,n)^{T_\lambda}$ in (\ref{eq:Brunion}).
 The elements of the group ${\rm GL}_d(\mathcal{O}_\KK)$ are
$d \times d$-matrices 
$ {\bf A} \, = \, A_0 + A_1 t + A_2 t^2 + A_3 t^3 + \cdots$,
where each $A_i$ is a $d \times d$ matrix with entries in $K$, and ${\rm det}(A_0) \not=0$.
This group acts naturally on (\ref{eq:ballr}) and on (\ref{eq:identification}).
The $d \times d$ matrix ${\bf A}$ with entries in $\mathcal{O}_\KK \subset \KK$
admits the following representation by an
$n \times n$ matrix over the residue field $K$:
\begin{equation}
\label{eq:bigAmatrix} A \,\, = \,\,
\begin{pmatrix}
 A_0 & A_1 & A_2 & \cdots & A_{r-2} & A_{r-1} \\
  0    & A_0 & A_1 & \cdots & A_{r-3} & A_{r-2} \\
  0    &  0   & A_0 & \cdots & A_{r-4} & A_{r-3} \\
  \vdots & \vdots & \vdots & \ddots &  \vdots & \vdots \\
  0  & 0 & 0 & \cdots & A_0 & A_1 \\
   0  & 0 & 0 & \cdots & 0 & A_0 \\
 \end{pmatrix}.
\end{equation}
To get this nice block form, the basis  of $K^n$ shown
in (\ref{eq:nicebasis}) has to be reordered as follows:
\begin{equation}
 \label{eq:otherbasis}
 e_1,e_2,\ldots,e_d,\,
 te_1, t e_2, \ldots, t e_d ,\,\ldots\ldots,\,
 t^{r-1}e_1,t^{r-1}e_2,\ldots ,t^{r-1} e_d. 
\end{equation}
The matrices $A$ act on  each of the components in (\ref{eq:Brunion}).
We are interested in their orbits.

Let $\mu $ be a partition of the integer $\ell$
with at most $d$ parts and largest part at most~$r$. To be precise,
we write $\mu = (\mu_1, \ldots, \mu_d) \in \NN^d$ where
$r \geq \mu_1  \geq \mu_2 \geq \cdots \geq \mu_d \geq 0$
and $\sum_{i=1}^d \mu_i = \ell$.
With this partition we associate the lattice
$\,L_\mu \, = \,t^{r-\mu_1} \mathcal{O}_\KK \, e_1 \,\oplus \,
t^{r-\mu_2} \mathcal{O}_\KK\, e_2 \,\oplus \, \cdots \,\, \oplus\,
t^{r-\mu_d} \mathcal{O}_\KK \, e_d $.
The corresponding subspace of $ K^n$
is spanned by standard basis vectors:
$$ L_\mu  \,\, = \,\, K \bigl\{
t^{r-i} e_j \,:\, 1 \leq i \leq \mu_j \,\, {\rm and} \,\, 1 \leq j \leq d \bigr\}. $$
By construction, we have $L_\mu \in {\rm Gr}(\ell,n)^T$.
Since $L_\mu$ is a coordinate subspace, its Pl\"ucker coordinates
are given by one of the basis points in $\PP^{\binom{n}{\ell}-1}$, here denoted
$e_\mu$ for simplicity.

The  orbit of $L_\mu$ under the above group action
is a constructible subset of ${\rm Gr}(\ell,n)^T \subset \PP^{\binom{n}{\ell}-1}$. 
It consists of all points $\,e_\mu \cdot \wedge_\ell A\,$ that represent the subspaces $\,L_\mu A$,
where $A$ runs over all matrices of the form (\ref{eq:bigAmatrix}).
Let $W_\mu$ denote the Zariski closure of this orbit. In symbols,
$$ W_\mu \,\, = \,\, \overline{{\rm GL}_d(\mathcal{O}_\KK)  \cdot L_\mu}
\quad \subseteq \,\, {\rm Gr}(\ell,n)^T. $$
The variety $W_\mu$ is called a \textit{Schubert variety}. We immediately obtain the following lemma.

\begin{remark}
\label{rem-schurr}
For each partition $\mu$ of $\ell$, the Schubert variety $W_\mu$ is irreducible.
It is given by an explicit polynomial parametrization, namely $\,A \mapsto e_\mu \cdot \wedge_\ell A$,
which encodes $A \mapsto L_\mu A$.
\end{remark}

\begin{example}[$d{=}3,r{=}2,\ell{=}3$]
Let $\mu = (2,1,0)$ with the basis (\ref{eq:otherbasis}) of $K^6$.
 The subspace 
$L_\mu$ corresponds to the point $e_\mu = e_{145}$ in
${\rm Gr}(3,6)^T \subset \PP^{19}$. Its image under $A$ is the row space~of
\begin{equation}
\label{eq:schubertex}
\begin{pmatrix} 
1 &   \! 0   &\!  0  &\,  0  &\!  0  &\!  0 \\
0 &  \! 0  & \! 0  &\,  1  & \! 0  &\!  0 \\
0  &  \! 0  & \! 0  &\,  0  &\!  1  &\!  0 
\end{pmatrix} \cdot A \,\, = \,\,
\begin{pmatrix}
 a_{011} & a_{012} & a_{013} & a_{111}  & a_{112}  &  a_{113} \\
 0  &     0  &     0   &   a_{011} &        a_{012}      &     a_{013}    \\
 0  &     0  &     0   &   a_{021} &        a_{022}      &    a_{023}    
\end{pmatrix}.
\end{equation}
The action of the group ${\rm GL}_3( \mathcal{O}_\KK) $ on $K^6$ is given by the matrix
in (\ref{eq:bigAmatrix}), here written as
$$  
A \,\,  = \,\, \begin{pmatrix} A_0 & A_1 \\ 0 & A_0 \end{pmatrix} \,\, = \,\,
\begin{small}
\begin{pmatrix}
     a_{011} &    a_{012} & a_{013} &    a_{111} & a_{112} & a_{113} \\
  a_{021} & a_{022} & a_{023} & a_{121} & a_{122} & a_{123} \\
  a_{031} & a_{032} & a_{033} & a_{131} & a_{132} & a_{133} \\
      0    &   0  &     0   & a_{011} & a_{012} & a_{013} \\
      0    &   0  &     0   & a_{021} & a_{022} & a_{023} \\
      0    &   0  &     0   & a_{031} & a_{032} & a_{033} 
      \end{pmatrix}.
      \end{small}
$$
The Schubert variety $W_\mu$ is parametrized by all matrices (\ref{eq:schubertex}).
As a subvariety of the Grassmannian ${\rm Gr}(3,6)$, it is defined by 
the following $11$ linear forms in the $20$ Pl\"ucker coordinates:
\begin{equation}
\label{eq:11shuffle}
p_{123}, \,\,p_{124}, \,p_{134},\, p_{234}, \,\,p_{125},\,
p_{135},\, p_{235}, \,\,p_{126},\, p_{136},\, p_{236},\,\,
p_{156}-p_{246}+p_{345}.
\end{equation}
This subvariety has dimension $4$ and degree $6$, and we find that $\, {\rm Gr}(3,6)^T = W_\mu$.
It is the entry $[11,4,6]$ for $\lambda = (2,2,2)$ of Table \ref{tab:67}.
The expressions (\ref{eq:11shuffle}) are the shuffle equations. \hfill
$\diamond$ \end{example}

The duality of Theorem \ref{th:duality} acts on the Schubert varieties as follows.
The complement to $\mu = (\mu_1,\mu_2,\ldots,\mu_d)$ is the
partition $\mu^c = (r-\mu_d,r-\mu_{d-1},\ldots,r-\mu_1)$ of the integer $n-\ell$.
Then the inclusion 
$W_{\mu^c} \subseteq {\rm Gr}(n-\ell,n)^T$ is isomorphic to 
the inclusion $W_\mu \subseteq {\rm Gr}(\ell,n)^T$.

\smallskip

We summarize the above discussion as follows:
for any partition $\mu$ of $\ell$ with $\leq d$ parts of size $\leq r$, we have constructed an irreducible 
subvariety $W_\mu$ of $\mathrm{Gr}(\ell,n)^T$.
Here $n=dr$ and $T = T_\lambda $ for $\lambda = (r,r,\ldots,r)$.
The union of these varieties equals $\mathrm{Gr}(\ell,n)^T$
because every lattice in the ball $\mathcal{B}_r$
lies in the ${\rm GL}(\mathcal{O}_\KK)$-orbit of some lattice
$\,L_\mu \, = \,t^{r-\mu_1} \mathcal{O}_\KK \, e_1 \,\oplus \, \cdots \,\, \oplus\,
t^{r-\mu_d} \mathcal{O}_\KK \, e_d $.

To proceed further, we record the following fact about inclusions
of Schubert varieties.

\begin{lemma}\label{lem:dominance}
Let $\mu,\nu$ be two partitions of $\ell$ with at most $d$ parts
and largest part at most $r$. Then
the inclusion $W_\mu \subseteq W_\nu$ holds if and only if
$\mu \leq \nu$ in the dominance order on partitions.
\end{lemma}

\begin{proof}
This is well known in algebraic combinatorics; see e.g. \cite[Remark 5.3.4]{LB15}.
\end{proof}

Deriving the irreducibility of ${\rm Gr}(\ell,n)^T$ is now reduced to a 
combinatorial argument.

\begin{proof}[Proof of Theorem \ref{thm:rec}]
Two partitions satisfy $\mu\leq \nu$ in dominance order if
and only if $\mu_1 + \cdots + \mu_i \leq \nu_1 + \cdots + \nu_i$ for all $i$.
Consider the set $\mathcal{P}$ of all partitions of $\ell$ with at most $d$ parts
whose largest part has size at most $r$. The restriction of dominance order
to this set has a unique largest element $\mu_{\rm max}$.
Namely, this largest partition equals $\mu_{\rm max} = (r,r,\ldots,r,b)$.
The partition $\mu_{\rm max}$
 has $a$ blocks of size $r$, where $\ell = ar + b$ and $ 0 \leq b < r$.
Lemma \ref{lem:dominance} implies
\begin{equation} \label{eq:W=Gr}
W_{\mu_{\rm max}} \,\,=\,\,
      \bigcup_{\mu\in\mathcal{P}} W_{\mu}\,\,=\,\,\mathrm{Gr}(\ell,n)^T.      
\end{equation}
In light of Remark \ref{rem-schurr}, this proves the irreducibility of
$\mathrm{Gr}(\ell,n)^T$, i.e.~Theorem \ref{thm:rec} holds.
\end{proof}

\begin{corollary} 
Fix $T = T_\lambda$ where $\lambda = (r,r,\ldots,r) $
and set $a = \lfloor \ell/r \rfloor$ and $b= \ell - ar$.
The dimension of the irreducible variety ${\rm Gr}(\ell,n)^T$ is equal to $(d-a)\ell - (a+1)b$.
\end{corollary}

\begin{proof}
We compute the dimension of $W_\mu$ for any $\mu \in \mathcal{P}$.
For the action of ${\rm GL}_d(\mathcal{O}_\KK)$ by the group of matrices $A$,
we determine the stabilizer of the distinguished point $L_\mu$.
Every matrix in this stabilizer has $A_1 = A_2 = \cdots = A_{r-1} = 0$.
The matrix $A_0$ breaks into blocks according to
various levels given by powers of $t$. This can be expressed
conveniently by the partition $\mu^* = (\mu^*_1,
\mu^*_2, \ldots, \mu^*_r)$ that is conjugate
to $\mu = (\mu_1,\mu_2,\ldots,\mu_d)$.
Here $\mu_i^*$ is the number of indices $j$ such that
$\mu_j \geq i$. Note that $\mu^*$ is a partition of $\ell$
with at most $r$ parts and largest part at most $d$. The desired
stabilizer is the product of the matrix groups ${\rm GL}_{\mu_i^*}(K)$, for 
$i=1,2,\ldots,r$. In particular, the dimension of the stabilizer is $\sum_{i=1}^r (\mu^*_i)^2$. This implies
$$ {\rm dim}(W_\mu) \,\,=\,\, d \ell - \sum_{i=1}^r (\mu^*_i)^2\,\,=\sum_{1\leq i\leq j\leq d}(\mu_i-\mu_j). $$
We learned the last identity from \cite[eqn~(26)]{Voll}.
We shall apply the middle formula to the
maximal partition $\mu = \mu_{\rm max} = (r,r,\ldots,r,b)$.
Its conjugate partition is $\mu^* = (a+1,\ldots,a+1,a,\ldots,a)$,
with $b$ blocks of size $a+1$ and $r-b$ blocks of size $a$.
The middle formula yields
\begin{equation}
\label{eq:dimW} {\rm dim}(W_\mu) 
\,=\, d \ell - b(a+1)^2 - (r-b)a^2\,=\,
(d-a) \ell - (a+1)b. 
\end{equation}
The assertion now follows from (\ref{eq:W=Gr}).
\end{proof}

We now present further evidence in favor of
Conjecture \ref{conj:rec}, beginning with the
computational results shown in Table \ref{tab:1012}.
For each table entry we verified that the shuffle equations
span the space of linear forms that vanish on ${\rm Gr}(\ell,n)^T$.
For all entries marked with a star, the {\tt Macaulay2} command  {\tt isPrime J} terminated
and proved that the shuffle ideal is prime.

\begin{table}[h]
\begin{center}
\begin{tabular}{||c c c c c||} 
\hline
 $\lambda$ & $\ell=2$ & $\ell=3$ & $\ell = 4$ & $\ell = 5$ \\ [0.3ex]   
\hline\hline
(3,3,3) & [27, 4, 6]${}^*$  & [57, 6, 90]${}^*$  & [99, 6, 90]${}^*$  & [99, 6, 90]${}^*$ \\
        (5,5) &   [41, 2, 2]${}^*$  & [112, 3, 6]${}^*$  & [194, 4, 24]${}^*$ & [220, 5, 120]${}^*$         \\  
        (2,2,2,2,2) & [20, 8, 70]${}^*$ & [70, 10, 1050]${}^*$   & [110, 12, 23100] & [152, 12, 23100] \\ 
        (6,6) &  [62, 2, 2]${}^*$ &  [212, 3, 6]${}^*$ &  [479, 4, 24] &  [760, 5, ??]  \\
 (4,4,4) &  [57, 4, 6]${}^*$ &  [193, 6, 90]${}^*$ &   [414, 8, 2520] & [711, 8, ??]  \\
(3,3,3,3) &  [50, 6, 20]${}^*$ & [156, 9, 1680] &  [399, 10, 8400] & [648, 11, ??]  \\
(2,2,2,2,2,2) &  [30, 10, 252]${}^*$ &  [130, 13, 18018] &  [270, 16, ??] & [492, 17, ??] \\        
          \hline
  \end{tabular}
  \caption{\label{tab:1012} Fixed point loci for rectangular partitions of $n=9,10,12$.}
  \end{center}
\end{table}	   

The extremal cases  $\lambda=(1,1,\dots,1)$ and $\lambda=(n)$
had been excluded from the definition of rectangular partition,
but it is worthwhile to consider these now. Conjecture \ref{conj:rec} holds
for both of these cases.
  Indeed for $\lambda=(1,1\dots,1)$, 
  we have $T = {\rm 0}_n$, so there are no shuffle equations.
  The corresponding variety $\mathrm{Gr}(\ell,n)^{T_\lambda}$
  agrees with the Grassmannian $\mathrm{Gr}(\ell,n)$.
  This is defined by the Pl\"ucker quadrics which are well-known to generate a prime ideal. 
  
  We conclude by addressing the case   $\lambda=(n)$. This was studied 
  for $n=6$ in   Example \ref{ex:(6)}. We now generalize what we saw there,
  namely Conjecture \ref{conj:rec} holds for the one-part partition.

\begin{proposition}
\label{conj:fullpart}
For $\lambda=(n)$ and any $\ell$, the shuffle equations are all Pl\"ucker coordinates $p_I$ except for $I=(n-\ell+1,\dots,n)$. These generate the prime ideal of the point $\,{\rm Gr}(\ell,n)^T = \{e_I\}$.
\end{proposition}

\begin{proof}
Consider any ordered $\ell$-set $I$ in $[n]$.
If $n \not\in I$ then $p_I$ equals the coefficient of $z^\ell$ in the
       coordinate of the row vector $ P \cdot \wedge_\ell(I+zT)$
that is indexed by $I+(1,1,\ldots,1,1)$.
       If $I = (J,n) $ but $ n-1 \not\in J$ then, modulo the above Pl\"ucker
       coordinates, $p_I $ equals the coefficient of $z^{\ell-1}$ of the
       coordinate in $ P \cdot \wedge_\ell(I+zT)$ that is indexed by $ I+(1,1,\ldots,1,0)$.
              If $I = (J,n-1,n)$ but $n-2 \not\in J$ then, modulo the above
       Pl\"ucker coordinates, $p_I$ equals the coefficient of $z^{\ell-2}$ of 
       the coordinate that is indexed by $ I+(1,...,1,0,0)$ etc...
                     Iterating this process yields all Pl\"ucker coordinates other
       than the last one, $I = (n-\ell+1, n-\ell+2, ...,n-1,n)$.
   For $n = 6$ and $\ell = 3$, our argument can be checked by
      looking at the      $20$ expressions in Example \ref{ex:(6)}.     
\end{proof}

	\bigskip \medskip
	
	{\bf Acknowledgements}.
	This project was supported by the Deutsche Forschungsgemeinschaft 
	(DFG, German Research Foundation) -- Project-ID 286237555 --TRR 195.
	We thank Yassine El Maazouz 
	and Mike Stillman for helpful communications about this project.

	\bigskip \medskip
	
	\noindent
	\footnotesize 
	{\bf Authors' addresses:}
	
	\smallskip
		
	\noindent Marvin Anas Hahn,   
	Trinity College Dublin
	\hfill {\tt Hahnma@maths.tcd.ie}
		
	\noindent Gabriele Nebe, RWTH Aachen
	\hfill {\tt gabriele.nebe@rwth-aachen.de}
	
	\noindent Mima Stanojkovski,  
	Universit\`a di Trento
	\hfill {\tt mima.stanojkovski@unitn.it}
		
	\noindent Bernd Sturmfels,
	MPI-MiS Leipzig and UC Berkeley 
	\hfill {\tt bernd@mis.mpg.de}


\medskip
	
\begin{thebibliography}{10}
\begin{small}
\setlength{\itemsep}{-0.6mm}


\bibitem{M2} D.~Grayson  and M.~Stillman:  Macaulay2, a software system 
for research in algebraic geometry, available at 
{\tt http://www.math.uiuc.edu/Macaulay2/}.

\bibitem{KLMW}
V.~Kreiman, V.~Lakshmibai, P.~Magyar and J.~Weyman:
{\em On ideal generators for~affine Schubert varieties},
  Algebraic groups and homogeneous spaces, 353--388, Tata Inst.~Fund. Res.~Stud.~Math., 19, 
   Mumbai, 2007. 
 
 \bibitem{LB15}
V.~Lakshmibai and L.~Brown:
{\em The Grassmannian Variety},
  Developments in Mathematics, Vol. 42, Springer, 2007. 
 
\bibitem{MuWeYa} D.~Muthiah, A.~Weekes and O.~Yacobi: {\em The equations defining affine
  Grassmannians in type A and a conjecture of Kreiman, Lakshmibai, Magyar, and Weyman},
  International Mathematics Research Notices {\bf 3}
  (2022) 1922-1972.
  

\bibitem{AIT}  
  B.~Sturmfels: {\em Algorithms in Invariant Theory}, Texts and Monographs in Symbolic Computation, 
  Springer-Verlag, Vienna, 1993
    
\bibitem{Voll} C.~Voll, {\em Functional equations for zeta functions of groups and rings}, 
Annals of Math.~{\bf 172} (2010) 1181--1218.

\end{small}

\end{thebibliography}
 \end{document}